\documentclass[a4paper,11pt]{article}
\usepackage{amsmath}
\usepackage{amsfonts}
\usepackage{amssymb}
\usepackage{amsthm}
\usepackage{calc}

\theoremstyle{plain}
\newtheorem{theorem}{Theorem}
\newtheorem{lemma}[theorem]{Lemma}

\newcommand{\lem}[1]{Lemma~\ref{lem:#1}}
\newcommand{\thm}[1]{Theorem~\ref{thm:#1}}

\newcommand{\sect}[1]{Section~\ref{sec:#1}}
\newcommand{\sys}[1]{System~\ref{sys:#1}}

\newcommand{\defn}{\emph}

\newcommand{\N}{\mathbb{N}}
\newcommand{\Z}{\mathbb{Z}}
\newcommand{\R}{\mathbb{R}}
\newcommand{\column}[1]{c^{(#1)}}

\newcommand{\vvdots}[1]{\mathrel{\makebox[\widthof{\ensuremath{#1}}]{\vdots}}}

\title{Partition regularity of a system \\of De and Hindman}
\author{Ben Barber\footnote{Department of Pure Mathematics and Mathematical Statistics, Centre for Mathematical Sciences, Wilberforce Road, Cambridge, CB3 0WB, UK.  {\tt b.a.barber@dpmms.cam.ac.uk}}}

\begin{document}

\maketitle

\begin{abstract}
We prove that a certain matrix, which is not image partition regular over $\R$ near zero, is image partition regular over $\N$.  This answers a question of De and Hindman.
\end{abstract}

\section{Introduction}

Let $A$ be an integer matrix with only finitely many non-zero entries in each row.  We call $A$ \defn{kernel partition regular} (\defn{over $\N$}) if, whenever $\N$ is finitely coloured, the system of linear equations $Ax=0$ has a monochromatic solution; that is, there is a vector $x$ with entries in $\N$ such that $Ax = 0$ and each entry of $x$ has the same colour. We call $A$ \defn{image partition regular} (\defn{over $\N$}) if, whenever $\N$ is finitely coloured, there is a vector $x$ with entries in $\N$ such that each entry of $Ax$ is in $\N$ and has the same colour.  We also say that the system of equations $Ax = 0$ or the system of expressions $Ax$ is \defn{partition regular}.

The finite partition regular systems of equations were characterised by Rado.  Let $A$ be an $m \times n$ matrix and let $\column{1}, \dotsc, \column{n}$ be the columns of $A$.  Then $A$ has the \defn{columns property} if there is a partition $[n] = I_1 \cup I_2 \cup \dotsb \cup I_t$ of the columns of $A$ such that $\sum_{i \in I_1} \column{i} = 0$, and, for each $s$,
\begin{equation*}
 \sum_{i \in I_s} \column{i} \in \langle \column{j} : j \in I_1 \cup \dotsb \cup I_{s-1}\rangle,
\end{equation*}
where $\langle \cdot \rangle$ denotes (rational) linear span.

\begin{theorem}[\cite{Rado}]
 A finite matrix $A$ with integer coefficients is kernel partition regular if and only if it has the columns property.
\end{theorem}

The finite image partition regular systems were characterised by Hindman and Leader \cite{finite-image}.

In the infinite case even examples of partition regular systems are hard to come by: see \cite{bhl} for an overview of what is known.  De and Hindman \cite[Q3.12]{De-Hindman} asked whether the following matrix was image partition regular over $\N$.
\[
\begin{pmatrix}
1 &  &  &  &  &  &  &   \cdots \\
 & 1 & 1 &  &  &  &  &   \cdots \\
2 & 1 &  &  &  &  &  &  \cdots \\
2 &  & 1 &  &  &  &  &   \cdots \\
 &  &  & 1 & 1 & 1 & 1 &   \cdots \\
4 &  &  & 1 &  &  &  &   \cdots \\
4 &  &  &  & 1 &  &  &   \cdots \\
4 &  &  &  &  & 1 &  &   \cdots \\
4 &  &  &  &  &  & 1 &   \cdots \\
\vdots & \vdots & \vdots & \vdots & \vdots & \vdots & \vdots & \ddots \\
\end{pmatrix}
\]
where we have omitted zeroes to make the block structure of the matrix more apparent.  This matrix corresponds to the following system of linear expressions.

\begin{align*}
 x_{21} & + x_{22}                      &   x_{21} & + 2 y       & y,  \\
 &&  x_{22} & + 2 y       &                                        &    \\
                                                                      \\
 x_{41} & + x_{42} + x_{43} + x_{44}    &   x_{41} & + 4 y       &    \\
                &                        &   x_{42} & + 4 y       &    \\
                &                        &   x_{43} & + 4 y       &    \\
                &                        &   x_{44} & + 4 y       &    \\
          &   &        & \vvdots{+}                    &    \\
 x_{2^n1} & + \cdots + x_{2^n2^n}             &   x_{2^n1} & + 2^n y &    \\
                                &        &          & \vvdots{+}  &    \\
                       &           &   x_{2^n2^n} & + 2^n y &    \\
          &   &        & \vvdots{+}                    &
\end{align*}

A matrix $A$ is called \defn{image partition regular over $\R$ near zero} if, for every $\delta > 0$, whenever $(-\delta, \delta)$ is finitely coloured, there is a vector $x$ with entries in $\R\setminus \{0\}$ such that each entry of $Ax$ is in $(-\delta, \delta)$ and has the same colour.  De and Hindman sought a matrix that was image partition regular over $\N$ but not image partition regular over $\R$ near zero.  It is easy to show that the above matrix is not image partition regular over $\R$ near zero, so showing that it is image partition regular over $\N$ would provide an example.

The main result of this paper is that De and Hindman's matrix is image partition regular over $\N$.

\begin{theorem}\label{thm:main}
For any sequence $(a_n)$ of integer coefficients, the system of expressions

\begin{align}\label{sys:new}
 x_{11} &                   &  x_{11} & + a_1 y  & y, \nonumber \\
                                                      \nonumber \\
 x_{21} & + x_{22}          &  x_{21} & + a_2 y  &    \nonumber \\
 &        &  x_{22} & + a_2 y                         \nonumber \\
                                                      \\
 x_{31} & + x_{32} + x_{33} &  x_{31} & + a_3 y  &    \nonumber \\
        &                   &  x_{32} & + a_3 y  &    \nonumber \\
        &                   &  x_{33} & + a_3 y  &    \nonumber \\
        &         &       & \vvdots{+} &    \nonumber
\end{align}

is partition regular.
\end{theorem}

Taking $a_n = n$ implies that De and Hindman's matrix is image partition regular.

Barber, Hindman and Leader \cite{bhl} recently found a different matrix that is image partition regular but not image partition regular over $\R$ near zero.  Their argument proceeded via the following result on kernel partition regularity.

\begin{theorem}[\cite{bhl}] \label{thm:bhl}
For any sequence $(a_n)$ of integer coefficients, the system of equations
\begin{equation*}
 \begin{split}
                  x_{11} + a_1 y & = z_1      \\
         x_{21} + x_{22} + a_2 y & = z_2      \\
                                 & \vvdots{=} \\
x_{n1} + \cdots + x_{nn} + a_n y & = z_n      \\
                                 & \vvdots{=}
 \end{split}
\end{equation*}
is partition regular.
\end{theorem}

In \sect{nearmiss} we show that \thm{main} can almost be deduced directly from \thm{bhl}.  The problem we encounter motivates the proof of \thm{main} that appears in \sect{main}.

\section{A near miss} \label{sec:nearmiss}

In this section we show that \thm{main} can almost be deduced directly from \thm{bhl}.

Let $\N$ be finitely coloured.  By \thm{bhl} there is a monochromatic solution to the system of equations
\begin{align*}
                         \tilde x_{11} -   a_1 y & = z_1      \\
         \tilde x_{21} + \tilde x_{22} - 2 a_2 y & = z_2      \\
                                                 & \vvdots{=} \\
\tilde x_{n1} + \cdots + \tilde x_{nn} - n a_n y & = z_n      \\
                                                 & \vvdots{=}
\end{align*}
Set $x_{ni} = \tilde x_{ni} - a_n y$.  Then, for each $n$ and $i$,
\[
 x_{n1} + \cdots + x_{nn} = \tilde x_{n1} + \cdots + \tilde x_{nn} - n a_n y = z_n,
\]
and
\[
 x_{ni} + a_n y = \tilde x_{ni},
\]
so we have found a monochromatic image for \sys{new}.  The problem is that we have not ensured that the variables $x_{ni} = \tilde x_{ni} - a_n y$ are positive.  In \sect{main} we look inside the proof of \thm{bhl} to show that we can take (most of) the $x_{ni}$ to be as large as we please.

\section{Proof of \thm{main}} \label{sec:main}

The proof of \thm{bhl} used a density argument.  The (\defn{upper}) \defn{density} of a set $S\subseteq \N$ is
\[
 d(S) = \limsup_{n \to \infty} \frac{|S \cap [n]|}{n},
\]
where $[n] = \{1,2,\ldots,n\}$.  The density of a set $S \subseteq \Z$ is $d(S \cap \N)$.  We call $S$ \defn{dense} if $d(S) > 0$.  We shall use three properties of density.
\begin{enumerate}
 \item If $A \subseteq B$, then $d(A) \leq d(B)$.
 \item Density is unaffected by translation and the addition or removal of finitely many elements.
 \item Whenever $\N$ is finitely coloured, at least one of the colour classes is dense.
\end{enumerate}

We will also use the standard notation for sumsets and difference sets
\begin{align*}
 A + B & = \{a + b : a \in A, b \in B\} \\
 A - B & = \{a - b : a \in A, b \in B\} \\
      kA & = \underbrace{A + \cdots + A}_{k \text{ times}},
\end{align*}
and write $m \cdot S=\{ms:s\in S\}$ for the set obtained from $S$ under pointwise multiplication by $m$.

We start with two lemmas from \cite{bhl}.

\begin{lemma}[\cite{bhl}] \label{lem:symmetric}
 Let $A \subseteq \N$ be dense.  Then there is an $m$ such that, for $n \geq 2/d(A)$, $nA-nA = m \cdot \Z$.
\end{lemma}

\begin{lemma}[\cite{bhl}] \label{lem:translated}
 Let $S \subseteq \Z$ be dense with $0 \in S$.  Then there is an $X \subseteq \Z$ such that,
for $n\geq 2/d(S)$, we have
$S-nS = X$.
\end{lemma}

The following consequence of Lemmas~\ref{lem:symmetric} and~\ref{lem:translated} is mostly implicit in \cite{bhl}.  The main new observation is that the result still holds if we insist that we only use large elements of $A$.  Write $A_{>t} = \{a \in A: a > t\}$.

\begin{lemma} \label{lem:new}
 Let $A$ be a dense subset of $\N$ that meets every subgroup of $\Z$, and let $m$ be the least common multiple of $1, 2, \ldots, \lfloor 1/d(A) \rfloor$.  Then, for $n \geq 2/d(A)$ and any $t$,
 \[
  A_{>t} - nA_{>t} \supseteq m \cdot \Z.
 \]
\end{lemma}

\begin{proof}
First observe that, for any $t$, $d(A_{>t}) = d(A)$.  Let $n \geq 2/d(A)$, and
let $X=A_{>t} - nA_{>t}$. For any $a \in A_{>t}$, we have by \lem{translated} that
\[
(A_{>t}-a) - n(A_{>t}-a) = (A_{>t}-a)-(n+1)(A_{>t}-a),
\]
and so
\[
X=X-A_{>t}+a.
\]
Since $a \in A_{>t}$ was arbitrary it follows that
$X=X+A_{>t}-A_{>t}$, whence $X=X+l(A_{>t}-A_{>t})$ for all $l$. By \lem{symmetric} there is an $m_t \in \Z$ such that, for $l \geq 2/d(A)$, $l(A_{>t}-A_{>t}) = m_t \cdot \Z$.  Hence $X = X + m_t \cdot \Z$, and $X$ is a union of cosets of $m_t \cdot \Z$.  Since $A$ contains arbitrarily large multiples of $m_t$, one of these cosets is $m_t \cdot \Z$ itself.

 Since $lA_{>t} - lA_{>t}$ contains a translate of $A_{>t}$,
 \[
  1/m_t = d(m_t \cdot \Z) \geq d(A),
 \]
 and $m_t \leq 1/d(A)$.  So $m_t$ divides $m$ and
 \[
  A_{>t} - nA_{>t} \supseteq m \cdot \Z. \qedhere
 \]
\end{proof}

\lem{new} will allow us to find a monochromatic image for all but a finite part of \sys{new}.  The remaining finite part can be handled using Rado's theorem, provided we take care to ensure that it gives us a solution inside a dense colour class.

\begin{lemma}[\cite{bhl}] \label{lem:ap}
 Let $\N$ be finitely coloured.  For any $l\in\N$, there is a $c \in \N$ such that $c \cdot [l]$ is disjoint from the non-dense colour classes.
\end{lemma}

We can now show that \sys{new} is partition regular.

\begin{proof}[Proof of \thm{main}]
 Let $\N$ be $r$-coloured.  Suppose first that some colour class does not meet every subgroup of $\Z$; say some class contains no multiple of $m$.  Then $m \cdot \N$ is $(r-1)$-coloured by the remaining colour classes, so by induction on $r$ we can find a monochromatic image.  So we may assume that every colour class meets every subgroup of $\Z$.

Let $d$ be the least density among the dense colour classes, and let $m$ be the least common multiple of $1, 2, \ldots, \lfloor 1/d \rfloor$.  Then for any dense colour class $A$, any $t$ and $n \geq 2/d$,
\[
 A_{>t} - nA_{>t} \supseteq m \cdot \Z.
\]

Now let $N = \lceil 2/d \rceil - 1$.  We will find a monochromatic image for the the expressions containing only $y$ and $x_{ni}$ for $n \leq N$ using Rado's theorem.  Indeed, consider the following system of linear equations.

\begin{align} \label{sys:finite}
 u_1 & = x_{11}                   &  v_{11} & = x_{11} + a_1 y  & y, \nonumber \\
                                                      \nonumber \\
 u_2 & = x_{21} + x_{22}          &  v_{21} & = x_{21} + a_2 y  &    \nonumber \\
 &        &                         v_{22} & = x_{22} + a_2 y  \nonumber \\
 &     &        & \vvdots{=}                        \\
 u_N & = x_{N1} + \cdots + x_{NN} &  v_{N1} & = x_{N1} + a_N y  &    \nonumber \\
        &                   &         & \vvdots{=}  &    \nonumber \\
        &                   &   v_{NN} & = x_{NN} + a_N y  &   \nonumber
\end{align}

The matrix corresponding to these equations has the form
\[
 \begin{pmatrix}
  B & -I
 \end{pmatrix}
\]
where $B$ is a top-left corner of the matrix corresponding to \sys{new} and $I$ is an appropriately sized identity matrix.  It is easy to check that this matrix has the columns property, so by Rado's theorem there is an $l$ such that, whenever a progression $c \cdot [l]$ is $r$-coloured, it contains a monochromatic solution to \sys{finite}.

Apply \lem{ap} to get $c$ with $c \cdot [ml]$ disjoint from the non-dense colour classes.  Then $mc \cdot [l] \subseteq c \cdot [ml]$ is also disjoint from the non-dense colour classes, and by the choice of $l$ there is a dense colour class $A$ such that $A \cap \left(md \cdot [l]\right)$ contains a solution to \sys{finite}.  Since the  $u_n$, $v_{ni}$ and $y$ are all in $A$, $y$ and the corresponding $x_{ni}$ make the first part of \sys{new} monochromatic.

Now $y$ is divisible by $m$, so for $n > N$ we have that
\[
 -n a_n y \in A_{>a_ny} - nA_{>a_ny},
\]
so there are $\tilde x_{ni}$ and $z_n$ in $A_{>a_ny}$ such that
\[
 -a_n y = z_n - \tilde x_{n1} - \cdots - \tilde x_{nn}.
\]
Set $x_{ni} = \tilde x_{ni} - a_n y$.  Then
\[
 x_{n1} + \cdots + x_{nn} = \tilde x_{n1} + \cdots + \tilde x_{nn} - n a_n y = z_n,
\]
and
\[
 x_{ni} + a_n y = \tilde x_{ni},
\]
for each $n > N$ and $1 \leq i \leq n$.  Since $\tilde x_{ni}$ and $z_n$ are in $A$ it follows that the whole of \sys{new} is monochromatic.

It remains only to check that all of the variables are positive.  But for $y$ and $x_{ni}$ with $n \leq N$ this is guaranteed by Rado's theorem; for $n > N$ it holds because $\tilde x_{ni} > a_n y$.
\end{proof}

\bibliography{dh}{}
\bibliographystyle{alpha}

\end{document}